\documentclass{amsart}
\usepackage{amsmath,amsfonts,amsthm}
\usepackage{amssymb,latexsym}
\usepackage{graphics}
\usepackage[colorlinks]{hyperref}

\usepackage[all]{xypic}

\theoremstyle{plain}
\theoremstyle{definition}
\newtheorem{theorem}{Theorem}[section]
\newtheorem{lemma}[theorem]{Lemma}

\newtheorem{corollary}[theorem]{Corollary}

\newtheorem{definition}[theorem]{Definition}
\newtheorem{example}[theorem]{Example}

\newtheorem{note}[theorem]{Note}

\theoremstyle{remark}

\numberwithin{equation}{section}
\newcommand{\SP}{\: \: \: \: \:}

\title{On Li--Yorke chaotic transformation groups modulo an ideal}
\author[M. Pourattar, F. Ayatollah Zadeh Shirazi]{Mehrnaz Pourattar, Fatemah Ayatollah Zadeh Shirazi}
\begin{document}
\begin{abstract}
In the following text we introduce the notion of chaoticity
modulo an ideal in the sense of Li--Yorke
in topological transformation semigroups and bring some of its elementary properties.
We continue our study by characterizing a class of abelian
infinite Li--Yorke chaotic Fort transformation groups and
show all of the elements of the above class is co--decomposable
to non--Li--Yorke chaotic transformation groups.
\end{abstract}
\maketitle
\noindent {\small {\bf 2010 Mathematics Subject Classification:}  54H15, 54H20 \\
{\bf Keywords:}} Fort space, Ideal, Li--Yorke chaos, Transformation semigroup.
\section{Introduction}
\noindent Different senses of chaos in dynamical systems like Devaney chaos~\cite{d1, dev, devaney}, Li--Yorke chaos~\cite{li-yorke},
distributional chaos~\cite{distributional}, $\omega$--chaos~\cite{w}, e--chaos~\cite{e}, ... for dynamical systems have been studied
in several texts, the main emphasis in these researches are on (compact) metric dynamical systems. Moreover,
recently have been done researches on chaos in transformation groups~\cite{trd}, maps on transformation groups~\cite{ctd}
and uniform phase spaces~\cite{ud}. On the other hand different compactifications (and amongst them one--point--compactification)
have their significant role in point set topology and topological dynamics~\cite{1, 2, 3}. In this text we present
a definition for Li--Yorke chaos in transformation semigroups (modulo an ideal)
with infinite phase semigroup and study this concept in the category of
transformation groups with one--point--compactification of a discrete space (i.e., a Fort space) as phase space.
\section{Preliminaries}
\noindent As it has been mentioned in Introduction in this text we deal with
Li--Yorke chaos in transformation semigroups with a uniform space as phase space, so we need backgrounds on
transformation semigroups, uniform spaces and Li--Yorke chaos, also we bring backgrounds on Fort spaces
too regarding our examples.
\subsection{Background on uniform spaces}\label{this}
Suppose $\mathcal F$ is a collection of subsets of $X\times X$ such that:
\begin{itemize}
\item $\forall\alpha\in\mathcal{F}\:\:(\Delta_X\subseteq\alpha)$,
\item $\forall\alpha,\beta\in\mathcal{F}\:\:(\alpha\cap\beta\in\mathcal{F})$,
\item $\forall\alpha\in\mathcal{F}\:\:\forall\beta\subseteq X\times X\:\:(\alpha\subseteq\beta\Rightarrow\beta\in\mathcal{F})$,
\item $\forall\alpha\in\mathcal{F}\:\:(\alpha^{-1}\in\mathcal{F})$,
\item $\forall\alpha\in\mathcal{F}\:\:\exists\beta\in\mathcal{F}\:\:(\beta\circ\beta\subseteq\alpha)$,
\end{itemize}
where $\Delta_X=\{(x,x):x\in X\}$ and $\alpha^{-1}=\{(y,x):(x,y)\in\alpha\}$ also $\alpha\circ\beta=\{(x,z):\exists y\:\:
((x,y)\in\beta\wedge(y,z)\in\alpha)\}$ (for $\alpha,\beta\subseteq X\times X$), then we call $\mathcal{F}$ a uniform
structure on $X$, also we call the elements of $\mathcal{F}$ entourages on $X$. For $\alpha\in\mathcal{F}$ and $x\in X$
let $\alpha[x]=\{y\in X:(x,y)\in \alpha\}$, then $\{U\subseteq X:\forall y\in U\:\:\exists\beta\in\mathcal{F}\:\:(\beta[y]\subseteq U)\}$
is a topology on $X$, we call it uniform topology on $X$ induced by uniform structure $\mathcal{F}$ and call $(X,\mathcal{F})$
or briefly $X$ a uniform space. We call the topological space $Y$ uniformzable if there exists a uniform structure $\mathcal{E}$ on
$Y$ such that uniform topology induced by $\mathcal{E}$ coincides with original topology on $Y$, also in this case
we say $\mathcal{E}$ is a compatible uniform structure on $Y$. Compact Hausdorff spaces are uniformzable
and admit a unique compatible uniform structure. 
In particular compact metric space in $(X,d)$
$\{\alpha\subseteq X\times X:\exists\varepsilon>0\:\:(O_\varepsilon\subseteq\alpha)\}$
is unique compatible uniform structure on $X$ (where $O_\varepsilon=\{(z,w)\in X\times X:d(z,w)<\varepsilon\}$
for every $\varepsilon>0$). 
For more details on uniform spaces see~\cite{dug, eng}.
\subsection{Ideals and Fort spaces}
Let's recall that we say the nonempty collection $\mathcal I$ of subsets of $W$ is an
ideal on $W$ if for all $A,B\in \mathcal{I}$ and $C$ with $C\subseteq A$ we have
$A\cup B,C\in\mathcal{I}$, in particular $\varnothing\in\mathcal{I}$. Although most of the authors in ideal $\mathcal{I}$ on
$W$ have supposed $X\notin \mathcal{I}$~\cite{set} we allow this condition too
(so $\mathcal{I}=\mathcal{P}(W)$ is allowed in this text, where $\mathcal{P}(W)=\{A:A\subseteq W\}$
is the collection of all subsets of $W$).
\\
Suppose $b\in F$ and equip $F$ with topology $\{U\subseteq F:b\notin U\vee(
F\setminus U$ is finite$)\}$, then we say $F$ is a Fort space with particular point $b$ (it's evident that Fort space $F$ with particular point $b$ is just
one point compactification (or Alexandroff compactification) of
discrete space $F\setminus\{b\}$) \cite{counter}.
\subsection{Background on Li--Yorke chaos in dynamical systems}
By a dynamical system $(X,f)$ we mean a topological space $X$ and continuous map
$f:X\to X$. In dynamical system $(X,f)$ with compact metric phase space $(X,d)$
we say $x,y\in X$ are
\begin{itemize}
\item[1.] proximal if
$\mathop{\liminf}\limits_{n\to\infty}d(f^n(x),f^n(y))=0$,
\item[2.] asymptotic if
$\mathop{\lim}\limits_{n\to\infty}d(f^n(x),f^n(y))=0$,
\item[3.] scrambled if
$\mathop{\liminf}\limits_{n\to\infty}d(f^n(x),f^n(y))=0$ and $\mathop{\limsup}\limits_{n\to\infty}d(f^n(x),f^n(y))>0$.
\end{itemize}
We say the dynamical
system $(X,f)$ is Li--Yorke chaotic if it has an uncountable subset like $A$ such that every
distinct $x,y\in A$ are scrambled. So for unique
compatible uniform structure on $X$, 
$\mathcal F=\{\alpha\subseteq X\times X:\exists\varepsilon>0\:\:(O_\varepsilon\subseteq\alpha)\}$, which is introduced in subsection~\ref{this}, we may use the following definitions too,
we say $x,y\in X$ are
\begin{itemize}
\item[1$^\prime$.] proximal if there exist $z\in X$ and net $\{n_\alpha\}_{\alpha\in\Gamma}$ in $\mathbb N$ with \[\mathop{\lim}\limits_{\alpha\in\Gamma}f^{n_\alpha}(x)=z=
\mathop{\lim}\limits_{\alpha\in\Gamma}f^{n_\alpha}(y)\:,\]
\item[2$^\prime$.] asymptotic if for every $\alpha\in{\mathcal F}$ the set $\{n\in\mathbb{N}:(f^n(x),f^n(y))\notin \alpha\}$ is finite,
\item[3$^\prime$.] scrambled if
 they are proximal and non--asymptotic.
 \end{itemize}
\subsection{Background on transformation semigroup}
By a transformation semigroup  (resp. transformation group)
$(X,S,\pi)$ or simply $(X,S)$ we mean a compact Hausdorff space $X$,
discrete topological semigroup (resp. group)  $S$ with identity $e$ and continuous map
$\pi:\mathop{X\times S\to X}\limits_{\:\:\:\:\:(x,s)\mapsto xs}$ such that
for all $x\in X$ and $s,t\in S$ we have $xe=x$, $x(st)=(xs)t$~\cite{ellis}. In particular, every dynamical system
$(X,f)$ may be considered as the transformation semigroup
$(X,\mathbb{N}\cup\{0\},\pi_f)$ where $\pi_f(x,n)=f^n(x)$
($x\in X,n\geq0$).
\\
We say $(X,(G_\alpha;\alpha\in\Gamma))$ is a multi---transformation semigroup (resp.  multi---transformation group) if for each $\alpha\in\Gamma$,
$(X,G_\alpha)$ is a transformation semigroup (resp.  transformation group), moreover for each distinct $\alpha_1,\ldots,\alpha_n\in\Gamma$ and $x\in X,s_1\in G_{\alpha_1},\ldots
s_n\in G_{\alpha_n}$ we have
\[(\cdots(xs_1)s_2)\cdots)s_n=(\cdots(xs_{\sigma(1)})s_{\sigma(2)})\cdots)s_{\sigma(n)}\]
for each permutation $\sigma:\{1,\ldots,n\}\to\{1,\ldots,n\}$.
\\
For transformation semigroup (resp.  transformation group) $(X,G)$, we say the multi--transformation semigroup (resp.  multi---transformation group) $(X,(G_\alpha;\alpha\in\Gamma))$
is a co--decomposition of $(X,G)$ if $G_\alpha$s are distinct sub--semigroups (resp. subgroups) of $G$, and $G$ is the semigroup (resp. group) generated by ${\displaystyle\bigcup_{\alpha\in\Gamma}G_\alpha}$ \cite{111}.
\begin{definition}\label{taha20}
In transformation semigroup $(X,S)$ with uniform phase space $(X,\mathcal{F})$ suppose $\mathcal I$ is an ideal on semigroup $S$.
We say $x,y\in X$ are:
\begin{itemize}
\item proximal if there exists $z\in X$ and a net $\{g_\alpha\}_{\alpha\in\Gamma}$ in $S$ with~\cite{ellis}
\[\mathop{\lim}\limits_{\alpha\in\Gamma}xg_\alpha=z=
\mathop{\lim}\limits_{\alpha\in\Gamma}yg_\alpha\:,\]
\item asymptotic modulo $\mathcal I$ if for every $\alpha\in{\mathcal F}$ we have $\{s\in S:(xs,ys)\notin \alpha\}\in\mathcal{I}$,
\item scrambled modulo $\mathcal I$ if
 they are proximal and non--asymptotic modulo $\mathcal I$,
\item $stab(x):=\{g\in S:xg=x\}$ is the stablizer of $x$.
 \end{itemize}
We denote the collection of all proximal pairs of $(X,S)$ with $Prox(X,S)$.
Moreover we have $Prox(X,S)=\bigcap\{\alpha S^{-1}:\alpha\in\mathcal{F}\}$ where for all $\alpha\in\mathcal{F}$
we have $\alpha S^{-1}=\{(z,w)\in X\times X:\exists s\in S\:((zs,ws)\in\alpha)\}$~\cite{gerko}.
Also we denote the collection of all asymptotic pairs $(z,w)$
modulo ideal $\mathcal I$ (i.e., $z,w\in X$ are asymptotic modulo ideal $\mathcal I$)
with $Asym_{\mathcal{I}}(X,S)$.
\\
Also we say $D\subseteq X$ with at least two elements is an scrambled set modulo $\mathcal I$
if for all distinct $z,w\in D$ we have $(z,w)\in Prox(X,S)\setminus Asym_{\mathcal{I}}(X,S)$.
We say $(X,S)$ is Li--Yorke chaotic  modulo $\mathcal I$ if it contains an uncountable scrambled subset  modulo $\mathcal I$.
\end{definition}
\begin{definition}\label{HM10}
In transformation semigroup $(X,S)$, $\mathcal{P}_{\rm fin}(S):=\{D\subseteq S:D\:{\rm is\: finite}\}$ is
an ideal on $S$, let
\[Asym(X,S):=Asym_{\mathcal{P}_{\rm fin}(S)}(X,S)\:.\]
We say
$(X,S)$ is Li--Yorke chaotic if it is Li--Yorke chaotic  modulo $\mathcal{P}_{\rm fin}(S)$. Also we say $x,y\in X$ are asymptotic
(resp. scrambled) if they are asymptotic modulo $\mathcal{P}_{\rm fin}(S)$ (resp. scrambled modulo $\mathcal{P}_{\rm fin}(S)$).
\end{definition}
\begin{note}
Consider dynamical system $(X,f)$ with compact metric phase space $X$ and transformation semigroup
 $(X,\mathbb{N}\cup\{0\})$ with $xn:=f^n(x)$ (for all $x\in X,n\geq0$), then $(X,f)$ is a Li--Yorke chaotic dynamical system if
 and only if  $(X,\mathbb{N}\cup\{0\})$ is a Li--Yorke chaotic
 transformation semigroup.
\end{note}
\begin{note}
For compact metric space $X$ with compatible metric $d$, and infinite countable semigroup
$S=\{t_1,t_2,\ldots\}$ (with distinct $t_n$s), in transformation semigroup $(X,S)$ the following statements are equivalent:
\begin{itemize}
\item[A.] $(X,S)$ is Li--Yorke chaotic (according to
Definition~\ref{HM10}),
\item[B.]  There exists an uncountable subset $A$ of $X$ such that for any distinct points $x,y\in A$ we have $(x,y)\in Prox(X,S)$ (i.e. there exists a sequence $\{s_n\}_{n\geq1}$ in $S$ with $\mathop{\lim}_{n\to\infty}d(xs_n,ys_n)=0$), and there exists 
$(r_n)_{n\geq1}\in\mathop{\prod}\limits_{n\geq1}S\setminus \{t_1,\ldots,t_n\}$ with $\mathop{\lim}\limits_{n\to\infty}d(xr_n,yr_n)>0$.
\item[C.]  for any increasing sequence $\mathfrak{F}=\{F_n\}_{n\geq1}$ of compact subsets of $S$ there exists an uncountable subset $A_{\mathfrak F}$ of $X$ such that for any distinct points $x,y\in A_{\mathfrak F}$ we have $(x,y)\in Prox(X,S)$, and there exists 
$(r_n)_{n\geq1}\in\mathop{\prod}\limits_{n\geq1}S\setminus F_n$ with $\mathop{\lim}\limits_{n\to\infty}d(xr_n,yr_n)>0$
(i.e., $(X,S)$ is Li--Yoke chaotic according to \cite[Definition 1.2]{dai}).
\end{itemize}
\end{note}
\begin{proof} 
Let's consider the following two claims for every $x,y\in X$:
\\
{\bf Claim 1.} If $(x,y)\notin Asym(X,S)$, then for any increasing sequence
 $\{F_n\}_{n\geq1}$ of finite subsets of $S$,
 there exists 
$(r_n)_{n\geq1}\in\mathop{\prod}\limits_{n\geq1}S\setminus F_n$ with $\mathop{\lim}\limits_{n\to\infty}d(xr_n,yr_n)>0$.
 \\
 {\it Proof of Claim 1.} Suppose $(x,y)\notin Asym(X,S)$, then there exists $\delta>0$
 such that $D:=\{s\in S:d(xs,ys)>\delta\}(=\{s\in S:(xs,ys)\notin O_\delta\})$ is infinite.
 Now consider increasing sequence $\{F_n\}_{n\geq1}$ of finite subsets of $S$,
 for all $n\geq1$ there exists $p_n\in D\setminus F_n$ also we may suppose $p_n$s are paiwise distinct, thus
 for all $n\geq1$, $d(xp_n,yp_n)>\delta$ which leads to
 $\varepsilon:=\mathop{\liminf}\limits_{n\to\infty}d(xp_n,yp_n)\geq\delta$,
 so $\{p_n\}_{n\geq1}$ has a subsequence $\{p_{n_k}\}_{k\geq1}$ with
  $\varepsilon=\mathop{\lim}\limits_{k\to\infty}d(xp_{n_k},yp_{n_k})>0$,
For all $k\geq1$ we have $n_k\geq k$ and $F_k\subseteq F_{n_k}$,
hence $p_{n_k}\in S\setminus F_{n_k}\subseteq S\setminus F_k$.  Thus
$(p_{n_k})_{k\geq1}\in\mathop{\prod}\limits_{k\geq1}S\setminus F_k$ which completes the proof of Claim 1.
\\
{\bf Claim 2.} If
 there exists 
$(r_n)_{n\geq1}\in\mathop{\prod}\limits_{n\geq1}S\setminus \{t_1,\ldots,t_n\}$ with $\mathop{\lim}\limits_{n\to\infty}d(xr_n,yr_n)>0$,
then  $(x,y)\notin Asym(X,S)$.
 \\
 {\it Proof of Claim 2.}
For all $n\geq1$ there exists $s_n\in S\setminus\{t_1,\ldots,t_n\}$ with $\varepsilon:=\mathop{\lim}\limits_{n\to\infty}d(xs_n,ys_n)>0$, so there exists
 $N\geq1$ with $d(xs_n,ys_n)>\varepsilon/2$ for all $n\geq N$ which leads to
 $\{s_n:n\geq N\}\subseteq \{s\in S:d(xs,ys)>\varepsilon/2\}$. 
 If $\{s_n:n\geq N\}$ is finite, then there exists $M\geq1$ with
 $\{s_n:n\geq N\}\subseteq \{t_1,\ldots,t_M\}$ in particular
 $s_{N+M}\in \{t_1,\ldots,t_M\}$ which is in contradiction with $s_{N+M}\in S\setminus\{t_1,\ldots,t_{N+M}\}$, hence 
 $\{s_n:n\geq N\}$ is infinte. Therefore $\{s\in S:d(xs,ys)>\varepsilon/2\}(=\{s\in S:(xs,ys)\notin O_{\varepsilon/2}\})$
 is infinite  too and $(x,y)\notin Asym(X,S)$.
\\
Now we are ready to prove the Note.
\\
``(A) $\Rightarrow$ (C)'' Use Claim 1 and the fact that the collection of finite subsets of $S$ is equal to the collection of compact subsets of $S$ (since $S$ is finite).
\\
``(C) $\Rightarrow$ (B)'' It is obvious, since $\{\{t_1,\ldots,t_n\}\}_{n\geq1}$ is an incresing sequence of compact subsets of $S$.
\\
``(B) $\Rightarrow$ (A)'' Use Claim 2.
\end{proof}
\section{Asymptoticity and Li--Yorke chaoticity modulo an ideal}
\noindent In this section we bring some elementary properties of
Li--Yorke chaoticity modulo an ideals in transformation semigroups,
in topics like products, quotient,
\linebreak
co--decomposition, .... in transformation semigroups.
\begin{theorem}
In transformation semigroup $(X,S)$ suppose $\mathcal{I}$ and $\mathcal{J}$
are ideals on $S$ with $\mathcal{I}\subseteq\mathcal{J}$. We have:
\\
$\bullet$ $Asym_{\mathcal{I}}(X,S)\subseteq Asym_{\mathcal{J}}(X,S)$,
\\
$\bullet$ if $D\subseteq X$ is an scrambled set modulo $\mathcal J$, then it is an scrambled set modulo $\mathcal I$,
\\
$\bullet$ if $(X,S)$ is Li--Yorke chaotic  modulo $\mathcal J$, then it is Li--Yorke chaotic  modulo $\mathcal I$.
\end{theorem}
\begin{proof}
Use the definition of asympoticity and Li--Yorke chaoticity modulo an ideal.
\end{proof}
\noindent In the transformation semigroup $(X,S)$ if $T$ is a sub--semigroup of $S$, then
we may consider transformation semigroup $(X,T)$ (with induced action of $S$ on $X$)
in a natural way too, in the following Theorem we deal with this type of
transformation semigroups.
\begin{theorem}
In transformation semigroup $(X,S)$ suppose $T$ is a sub--semigroup of $S$ and $\mathcal{I}$
is an ideals on $T$, then:
\begin{itemize}
\item[1.] $Asym_{\mathcal{I}}(X,S)\subseteq Asym_{\mathcal{I}}(X,T)$,
\item[2.] if $D\subseteq X$ is an scrambled set modulo $\mathcal I$ in $(X,T)$, then it is an scrambled set modulo $\mathcal I$ in $(X,S)$,
\item[3.] if $(X,T)$ is Li--Yorke chaotic  modulo $\mathcal I$, then $(X,S)$ is Li--Yorke chaotic  modulo $\mathcal I$,
\item[4.] if $(X,S)$ is co--decomposable to Li--Yorke chaotic  modulo $\mathcal I$ transformation semigroups if and only if it is Li--Yorke chaotic  modulo $\mathcal I$ (so with phase semigroups all of them containing $\bigcup\mathcal{I}$).
\end{itemize}
\end{theorem}
\begin{proof} First of all note that $\mathcal I$ is an ideal on $S$. Consider compatible uniform
structure $\mathcal F$ on $X$.
\\
1) For $x,y\in X$ we have (use $\{s\in T:(xs,ys)\notin U\}\subseteq \{s\in S:(xs,ys)\notin U\}$):
\begin{eqnarray*}
(x,y)\in Asym_{\mathcal{I}}(X,S) & \Rightarrow & (\forall U\in\mathcal{F}\:\{s\in S:(xs,ys)\notin U\}\in\mathcal{I}) \\
& \Rightarrow & (\forall U\in\mathcal{F}\:\{s\in T:(xs,ys)\notin U\}\in\mathcal{I}) \\
& \Rightarrow & (x,y)\in Asym_{\mathcal{I}}(X,T)\:.
\end{eqnarray*}
2) Use item (1) and $Prox(X,T)\subseteq Prox(X,S)$.
\\
3) Use item (2).
\\
4) If $(X,S)$ is Li--Yorke chaotic  modulo $\mathcal I$, then $(X,S)$ is a co--decomposition of itself to Li--Yorke chaotic  modulo $\mathcal I$ transformation semigroups. On the other
hand if $(X,(S_\alpha:\alpha\in\Gamma))$ is co--decomposition
of $(X,S)$ to Li--Yorke chaotic  modulo $\mathcal I$ transformation semigroups such that for all $\alpha\in\Gamma$
we have $\bigcup\mathcal{I}\subseteq S_\alpha$, then choose
$\alpha_0\in\Gamma$. Since $(X,S_{\alpha_0})$ is  Li--Yorke chaotic  modulo $\mathcal I$, $S_{\alpha_0}$ is a subsemigroup of $S$ and $\mathcal I$
is an ideal on $S_{\alpha_0}$ too, then $(X,S)$ is Li--Yorke chaotic  modulo $\mathcal I$ by item~(2).
\end{proof}
\noindent In transformation semigroup $(X,S)$ we say nonempty subset $Y$ of $X$ is invariant
if $YS:=\{ys:y\in Y,s\in S\}\subseteq Y$. If $Y$ is a closed invariant subset of $X$
then we may consider transformation semigroup $(Y,S)$ with induced action of $S$ on $X$.
\begin{note}
In transformation semigroup $(X,S)$ suppose $Y$ is a closed invariant subset of $X$
and $\mathcal{I}$ is an ideal on $S$, then
\\
$\bullet$ $Asym_{\mathcal{I}}(Y,S)\subseteq Asym_{\mathcal{I}}(X,S)$,
\\
$\bullet$ if $D\subseteq Y$ is an scrambled set modulo $\mathcal I$ in $(X,S)$, then it is an scrambled set modulo $\mathcal I$ in $(Y,S)$,
\end{note}
\noindent In the following Theorem we deal; with product of transformation semigroups.
\begin{theorem}\label{salam80}
Suppose $\{(X_\alpha,S):\alpha\in\Gamma\}$ is a nonempty set of transformation semigroups
and $\mathcal{I}$ is an ideal on $S$.
In transformation semigroup $(\mathop{\prod}\limits_{\alpha\in\Gamma}X_\alpha,S)$
with
\[(x_\alpha)_{\alpha\in\Gamma}s:=(x_\alpha s)_{\alpha\in\Gamma}\:\:\:\:\:
((x_\alpha)_{\alpha\in\Gamma}\in\mathop{\prod}\limits_{\alpha\in\Gamma}X_\alpha,
s\in S)\]
we have:
\\
1. $Asym_{\mathcal I}(\mathop{\prod}\limits_{\alpha\in\Gamma}X_\alpha,S)=\{((z_\alpha)_{\alpha\in\Gamma},(w_\alpha)_{\alpha\in\Gamma}):\forall
\alpha\in\Gamma\:((z_\alpha,w_\alpha)\in Asym_{\mathcal I}(X_\alpha,S))\}$,
\\
2. if $(z_\alpha)_{\alpha\in\Gamma},(w_\alpha)_{\alpha\in\Gamma}$ are scrambled modulo $\mathcal{I}$
(in transformation semigroup $(\mathop{\prod}\limits_{\alpha\in\Gamma}X_\alpha,S)$), then there exists
$\beta\in\Gamma$ such that $z_\beta,w_\beta$ are scrambled modulo $\mathcal{I}$
(in transformation semigroup $(X_\beta,S)$),
\\
3. for $\beta\in\Gamma$ suppose $p,q\in X_\beta$ and for each $\alpha\in\Gamma$ choose $z_\alpha\in X_\alpha$,
let
\[x_\alpha:=\left\{\begin{array}{lc} p & \alpha=\beta \:, \\ z_\alpha & \alpha\neq\beta \:, \end{array} \right. \:\:\:\:
y_\alpha:=\left\{\begin{array}{lc} q & \alpha=\beta \:, \\ z_\alpha & \alpha\neq\beta \:, \end{array} \right.\]
then $(x_\alpha)_{\alpha\in\Gamma},(y_\alpha)_{\alpha\in\Gamma}$ are scrambled modulo $\mathcal{I}$
(in transformation semigroup
\linebreak
$(\mathop{\prod}\limits_{\alpha\in\Gamma}X_\alpha,\mathop{\prod}\limits_{\alpha\in\Gamma}S_\alpha)$),
if and only if
$p,q$ are scrambled modulo $\mathcal{I}$
(in transformation semigroup $(X_\beta,S)$),
\\
4. if there exists $\beta\in\Gamma$ such that $(X_\beta,S)$ is Li--Yorke chaotic modulo $\mathcal I$,
then $(\mathop{\prod}\limits_{\alpha\in\Gamma}X_\alpha,S)$ is Li--Yorke chaotic modulo $\mathcal I$,
\end{theorem}
\begin{proof}
{\bf 1)} For compact Hausdorff topological space $Y$ suppose $\mathcal{F}_Y$ is the unique compatible
uniform structure on $Y$. For $\beta\in\Gamma$ and $U\in\mathcal{F}_{X_\beta}$ let:
\[M_\beta(U):=\{((z_\alpha)_{\alpha\in\Gamma},(w_\alpha)_{\alpha\in\Gamma})\in
\mathop{\prod}\limits_{\alpha\in\Gamma}X_\alpha\times\mathop{\prod}\limits_{\alpha\in\Gamma}X_\alpha:
(z_\beta,w_\beta)\in U\}\:.\]
Now suppose $((z_\alpha)_{\alpha\in\Gamma},(w_\alpha)_{\alpha\in\Gamma})\in
Asym_{\mathcal I}(\mathop{\prod}\limits_{\alpha\in\Gamma}X_\alpha,S)$, thus for
each $\beta\in\Gamma$ and $U\in \mathcal{F}_{X_\beta}$ (use $M_\beta(U)\in\mathcal{F}_{
\mathop{\prod}\limits_{\alpha\in\Gamma}X_\alpha})$) we have
\[\{s\in S:(z_\beta s,w_\beta s)\notin U\}=\{s\in S:((z_\alpha s)_{\alpha\in\Gamma},(w_\alpha s)_{\alpha\in\Gamma})\notin M_\beta(U)\}\in\mathcal{I}\]
which leads to $(z_\beta,w_\beta)\in Asym_{\mathcal I}(X_\beta,S)$. Therefore:
\[Asym_{\mathcal I}(\mathop{\prod}\limits_{\alpha\in\Gamma}X_\alpha,S)\subseteq\{((z_\alpha)_{\alpha\in\Gamma},(w_\alpha)_{\alpha\in\Gamma}):\forall
\alpha\in\Gamma\:((z_\alpha,w_\alpha)\in Asym_{\mathcal I}(X_\alpha,S))\}\:.\]
Now suppose for each $\alpha\in\Gamma$ we have $(p_\alpha,q_\alpha)\in Asym_{\mathcal I}(X_\alpha,S)$ and
$A\in\mathcal{F}_{\mathop{\prod}\limits_{\alpha\in\Gamma}X_\alpha}$. There exist
$\alpha_1,\ldots,\alpha_n\in\Gamma$ and $U_1\in\mathcal{F}_{X_{\alpha_1}},\ldots,U_n\in\mathcal{F}_{X_{\alpha_n}}$
with
\[\mathop{\bigcap}\limits_{1\leq i\leq n}M_{\alpha_i}(U_i)\subseteq A\: .\tag{*}\]
For each $i\in\{1,\ldots,n\}$ we have $(p_{\alpha_i},q_{\alpha_i})\in Asym_{\mathcal I}(X_{\alpha_i},S)$,
thus $\{s\in S:(p_{\alpha_i}s,q_{\alpha_i}s)\notin U_i\}\in\mathcal{I}$, so:
\[\mathop{\bigcup}\limits_{1\leq i\leq n}\{s\in S:(p_{\alpha_i}s,q_{\alpha_i}s)\notin U_i\}\in\mathcal{I}\tag{**}\]
thus:

$\{s\in S:((p_\alpha s)_{\alpha\in\Gamma},(q_\alpha s)_{\alpha\in\Gamma})\notin A\}$
\begin{eqnarray*}
& \mathop{\subseteq}\limits^{(*)} &
	\{s\in S:((p_\alpha s)_{\alpha\in\Gamma},(q_\alpha s)_{\alpha\in\Gamma})\notin \mathop{\bigcap}\limits_{1\leq i\leq n}M_{\alpha_i}(U_i)\} \\
& = & \mathop{\bigcup}\limits_{1\leq i\leq n} \{s\in S:((p_\alpha s)_{\alpha\in\Gamma},(q_\alpha s)_{\alpha\in\Gamma})\notin M_{\alpha_i}(U_i)\} \\
& = & \mathop{\bigcup}\limits_{1\leq i\leq n} \{s\in S:(p_{\alpha_i} s,q_{\alpha_i} s)\notin U_i\} \mathop{\in}\limits^{(**)}\mathcal{I}
\end{eqnarray*}
which shows $\{s\in S:((p_\alpha s)_{\alpha\in\Gamma},(q_\alpha s)_{\alpha\in\Gamma})\notin A\}\in\mathcal{I}$ and
$((p_\alpha s)_{\alpha\in\Gamma},(q_\alpha s)_{\alpha\in\Gamma})\in Asym_{\mathcal I}(\mathop{\prod}\limits_{\alpha\in\Gamma}X_\alpha,S)$.
Therefore:
\[Asym_{\mathcal I}(\mathop{\prod}\limits_{\alpha\in\Gamma}X_\alpha,S)\supseteq\{((z_\alpha)_{\alpha\in\Gamma},(w_\alpha)_{\alpha\in\Gamma}):\forall
\alpha\in\Gamma\:((z_\alpha,w_\alpha)\in Asym_{\mathcal I}(X_\alpha,S))\}\:.\]
{\bf 2)}  Use
$Prox(\mathop{\prod}\limits_{\alpha\in\Gamma}X_\alpha,S)\subseteq\{((z_\alpha)_{\alpha\in\Gamma},(w_\alpha)_{\alpha\in\Gamma}):\forall
\alpha\in\Gamma\:((z_\alpha,w_\alpha)\in Prox(X_\alpha,S))\}$
and item (1).
\\
{\bf 3)} If $p,q$ are scrambled modulo $\mathcal{I}$
in transformation semigroup $(X_\beta,S)$,
then by item (2), 
then $(x_\alpha)_{\alpha\in\Gamma},(y_\alpha)_{\alpha\in\Gamma}$ are scrambled modulo $\mathcal{I}$
in transformation semigroup
$(\mathop{\prod}\limits_{\alpha\in\Gamma}X_\alpha,\mathop{\prod}\limits_{\alpha\in\Gamma}S_\alpha)$. 
\\
Now suppose $(x_\alpha)_{\alpha\in\Gamma},(y_\alpha)_{\alpha\in\Gamma}$ are scrambled modulo $\mathcal{I}$
in transformation semigroup
$(\mathop{\prod}\limits_{\alpha\in\Gamma}X_\alpha,\mathop{\prod}\limits_{\alpha\in\Gamma}S_\alpha)$, then
by item (2) there exists $\alpha\in\Gamma$ such that $x_\alpha,y_\alpha$ are scrambled modulo $\mathcal{I}$
in transformation semigroup $(X_\alpha,S)$. If $\alpha\neq\beta$, then $(x_\alpha,y_\alpha)=(z_\alpha,z_\alpha)\in \Delta_{X_\alpha}\subseteq Asym_{\mathcal I}(X_\alpha,S)$ which is a contradiction to the fact that $x_\alpha,y_\alpha$ are scrambled modulo $\mathcal{I}$
and hence non--asymptotic modulo $\mathcal{I}$, therefore $\alpha=\beta$ and $p(=x_\beta),q=(y_\beta)$ are scrambled modulo $\mathcal{I}$.
\\
{\bf 4)} Use (2). 
\end{proof}
\begin{corollary}
Suppose $\{(X_\alpha,S_\alpha):\alpha\in\Gamma\}$ is a nonempty set of transformation semigroups
and for each $\alpha\in\Gamma$, $\mathcal{I}_\alpha$ is an ideal on $S_\alpha$.
In transformation semigroup $(\mathop{\prod}\limits_{\alpha\in\Gamma}X_\alpha,\mathop{\prod}\limits_{\alpha\in\Gamma}S_\alpha)$
with
\[(x_\alpha)_{\alpha\in\Gamma}(s_\alpha)_{\alpha\in\Gamma}:=(x_\alpha s_\alpha)_{\alpha\in\Gamma}\:\:\:\:\:
((x_\alpha)_{\alpha\in\Gamma}\in\mathop{\prod}\limits_{\alpha\in\Gamma}X_\alpha,
(s_\alpha)_{\alpha\in\Gamma}\in\mathop{\prod}\limits_{\alpha\in\Gamma}S_\alpha)\]
for each $\beta\in\Gamma$ and $D\in\mathcal{I}_\beta$ let
$H_\beta(D)=\{(s_\alpha)_{\alpha\in\Gamma}\in\mathop{\prod}\limits_{\alpha\in\Gamma}S_\alpha:s_\beta\in D\}$
and suppose $\mathcal{I}$ is an ideal on $\mathop{\prod}\limits_{\alpha\in\Gamma}S_\alpha$
generated by $\{H_\alpha(D):\alpha\in\Gamma,D\in\mathcal{I}_\alpha\}$. Also suppose $\mathcal R$
is an ideal on $\mathop{\prod}\limits_{\alpha\in\Gamma}S_\alpha$.
Then
we have:
\\
1.  $Asym_{\mathcal I}(\mathop{\prod}\limits_{\alpha\in\Gamma}X_\alpha,\mathop{\prod}\limits_{\alpha\in\Gamma}S_\alpha)$ is the set
\[\{((z_\alpha)_{\alpha\in\Gamma},(w_\alpha)_{\alpha\in\Gamma}):\forall
\alpha\in\Gamma\:((z_\alpha,w_\alpha)\in Asym_{{\mathcal I}_\alpha}(X_\alpha,S_\alpha))\}\:,\]
\\
2. if $(z_\alpha)_{\alpha\in\Gamma},(w_\alpha)_{\alpha\in\Gamma}$ are scrambled modulo $\mathcal{I}$
(in transformation semigroup
$(\mathop{\prod}\limits_{\alpha\in\Gamma}X_\alpha,\mathop{\prod}\limits_{\alpha\in\Gamma}S_\alpha)$),
then there exists
$\beta\in\Gamma$ such that $z_\beta,w_\beta$ are scrambled modulo $\mathcal{I}_\beta$
(in transformation semigroup $(X_\beta,S_\beta)$),
\\
3. with the same $(x_\alpha)_{\alpha\in\Gamma},(y_\alpha)_{\alpha\in\Gamma}$ as in item (3) of Theorem~\ref{salam80},
$(x_\alpha)_{\alpha\in\Gamma},(y_\alpha)_{\alpha\in\Gamma}$ are scrambled modulo $\mathcal{I}$
(in transformation semigroup
$(\mathop{\prod}\limits_{\alpha\in\Gamma}X_\alpha,\mathop{\prod}\limits_{\alpha\in\Gamma}S_\alpha)$),
if and only if
$p,q$ are scrambled modulo $\mathcal{I}_\beta$
(in transformation semigroup $(X_\beta,S_\beta)$),
\\
4. if there exists $\beta\in\Gamma$ such that $(X_\beta,S_\beta)$ is Li--Yorke chaotic modulo ${\mathcal I}_\beta$,
then $(\mathop{\prod}\limits_{\alpha\in\Gamma}X_\alpha,\mathop{\prod}\limits_{\alpha\in\Gamma}S_\alpha)$
is Li--Yorke chaptic modulo $\mathcal I$,
\end{corollary}
\begin{proof}
Use a similar method described in Theorem~\ref{salam80} and
$Prox(\mathop{\prod}\limits_{\alpha\in\Gamma}X_\alpha,\mathop{\prod}\limits_{\alpha\in\Gamma}S_\alpha)=\{((z_\alpha)_{\alpha\in\Gamma},(w_\alpha)_{\alpha\in\Gamma}):\forall
\alpha\in\Gamma\:((z_\alpha,w_\alpha)\in Prox(X_\alpha,S_\alpha))\}$.
\end{proof}
\begin{note}\label{taha50}
In transformation semigroups $(X,S),(Y,S)$ suppose $\varphi:(X,S)\to(Y,S)$ is a
homomorphism and $\mathcal{I}$ is an ideal of $S$, then
for $\mathop{\varphi\times\varphi:X\times X\to Y\times Y}\limits_{\:\:\:\:\:\:\:\:\:\:\:\:\:\:\:(x,y)\mapsto(\varphi(x),\varphi(y))}$ we have 
$\varphi\times\varphi(Prox(X,S))\subseteq Prox(Y,S)$~\cite{ellis},
and $\varphi\times\varphi(Asym_{\mathcal{I}}(X,S))\subseteq Asym_{\mathcal{I}}(Y,S)$,
suppose $(x,y)\in Asym_{\mathcal{I}}(X,S)$ and $U$ is an entourage of $Y$, since
$\varphi:X\to Y$ is continuous and $X,Y$ compact Hausdorff spaces,
$\varphi:X\to Y$ is uniformly continuous too. Thus there exists
entourage $V$ of $X$ with $\varphi\times\varphi(V)\subseteq U$. Using
$(x,y)\in Asym_{\mathcal{I}}(X,S)$ and $\varphi(zs)=\varphi(z)s$ for all $z\in X,s\in S$, we have:
\begin{eqnarray*}
\{s\in S:(\varphi(x)s,\varphi(y)s)\notin U\} & = & \{s\in S:(\varphi(xs),\varphi(ys))\notin U\} \\
& \subseteq & \{s\in S:(xs,ys)\notin V\}\in\mathcal{I}\:,
\end{eqnarray*}
therefore $\{s\in S:(\varphi(x)s,\varphi(y)s)\notin U\}\in\mathcal{I}$ and $(\varphi(x),\varphi(y))\in Asym_{\mathcal{I}}(Y,S)$.
\end{note}
\noindent In transformation semigroup $(X,S)$ suppose $\Re$ is a closed invariant
relation on $X$, then one may consider transformation semigroup $(\frac{X}{\Re},S)$~\cite{ellis, sabbagh}.
Using Note~\ref{taha50} and natural quotient homomorphism
$\pi_\Re:(X,S)\to(\frac{X}{\Re},S)$
we have the following Corollary.
\begin{corollary}
In transformation semigroup $(X,S)$ suppose $\Re$ is a closed invariant
relation on $X$ and $\mathcal{I}$ is an ideal on $S$, then
$\pi_\Re\times\pi_\Re (Asym_{\mathcal{I}}(X,S))\subseteq Asym_{\mathcal{I}}(\frac{X}{\Re},S)$.
\end{corollary}
\noindent Let's recall that in transformation semigroup $(X,S)$ with compatible uniform
structure $\mathcal{F}$ on $X$ for all $\alpha\in\mathcal{F}$ let $\alpha S^{-1}:=\{(z,w)\in X\times X:\exists s\in S\:(zs,ws)=(x,y)\}$,
then $Prox(X,S)=\bigcap\{\alpha S^{-1}:\alpha\in\mathcal{F}\}$~\cite{gerko}.
\begin{theorem}
In transformation semigroup $(X,S)$ with ${\rm card}(S)\geq2$ we have:
\begin{center}
$Prox(X,S)=\bigcup\{Asym_{\mathcal I}(X,S):{\mathcal I}$ is an ideal on $S$ with $\mathcal{I}\neq\mathcal{P}(S)\}$.
\end{center}
\end{theorem}
\begin{proof}
For ideal $\mathcal I$ on $S$ with $\mathcal{I}\neq\mathcal{P}(S)$ suppose $(x,y)\in Asym_{\mathcal I}(X,S)$
and $\mathcal{F}$ is the compatible uniform structure on $X$.
For every $\alpha\in\mathcal{F}$, we have $\{s\in S:(xs,ys)\notin\alpha\}\in\mathcal{I}$, thus
$\{s\in S:(xs,ys)\notin\alpha\}\neq S$ and there exists $s\in S$ with $(xs,ys)\in\alpha$, so
$(x,y)\in \alpha S^{-1}$. Therefore $(x,y)\in\bigcup\{\alpha S^{-1}:\alpha\in\mathcal{F}\}=Prox(X,S)$.
\\
On the other hand suppose $(x,y)\in Prox(X,S)$, thus $(x,y)\in\bigcap\{\alpha S^{-1}:\alpha\in\mathcal{F}\}$
and for every $\alpha\in\mathcal{F}$, there exists $s\in S$ with $(xs,ys)\in\alpha$ so
$J_\alpha:=\{t\in S:(xt,yt)\notin\alpha\}\neq S$. Let $\mathcal{I}:=\{A\subseteq S:\exists\alpha\in\mathcal{F}\:(A\subseteq J_\alpha)\}$.
For each $\alpha,\beta\in\mathcal{F}$ we have $\alpha\cap\beta\in\mathcal{F}$ and
$J_\alpha\cup J_\beta=J_{\alpha\cap\beta}$, thus $\mathcal{I}$ is an ideal on $S$ and $(x,y)\in Asym_{\mathcal I}(X,S)$. Moreover
for all $\alpha\in\mathcal{F}$ we have $J_\alpha\neq S$ thus $S\notin\mathcal{I}$ and
$\mathcal{I}\neq\mathcal{P}(S)$.
\end{proof}
\begin{note}\label{taha10}
In transformation semigroup $(X,S)$ suppose $\mathcal I$ is an ideal on $S$, being asymptotic modulo $\mathcal{I}$
is an equivalence relation on $X$, since if $x,y$ are asymptotic modulo $\mathcal{I}$ and $y,z$ are asymptotic modulo $\mathcal{I}$,
then for each $\alpha\in\mathcal{F}_X$ there exists $\beta\in{\mathcal F}_X$ with $\beta\circ\beta\subseteq\alpha$ and
we have $\{t\in S:(xt,yt)\notin\beta\},\{t\in S:(yt,zt)\notin\beta\}\in\mathcal{I}$ thus
$\{t\in S:(xt,zt)\notin\alpha\}\subseteq\{t\in S:(xt,yt)\notin\beta\}\cup\{t\in S:(yt,zt)\notin\beta\}\in\mathcal{I}$
which leads to $\{t\in S:(xt,zt)\notin\alpha\}\in\mathcal{I}$. Hence $x,z$ are asymptotic modulo $\mathcal{I}$ too.
\end{note}
\section{Li--Yorke chaotic Fort transformation groups}
\noindent In this section suppose $F$ is an infinite Fort space with particular point $b$. For each $D\subseteq F$ let: \[\alpha_D:=((F\setminus D)\times(F\setminus D))\cup\{(z,z):z\in D\}
\:,\]
then
\begin{center}
$\mathcal{K}:=\{U\subseteq F\times F:$ there exists finite subset $D\subseteq F\setminus\{b\}$ with $\alpha_D\subseteq U\}$
\end{center}
is the unique compatible uniform structure of $F$.
\begin{lemma}\label{salam10}
In infinite Fort transformation group $(F,G)$ we have:
\\
1) $\{(b,x):xG{\rm \: is \: infinite}\}\cup\{(x,b):xG{\rm \: is \: infinite}\}\subseteq Prox(F,G)$.
\\
2) For
\begin{eqnarray*}
P & := & \{(x,x):x\in F\}\cup \\
& & \{(b,x):xG{\rm \: is \: infinite}\}\cup\{(x,b):xG{\rm \: is \: infinite}\}\cup  \\
& & \{(x,y):xG{\rm \: and \:} yG{\rm \: are \: infinite}\}
\end{eqnarray*}
we have $Prox(F,G)\subseteq P$.
\\
3) Moreover if $G$ is abelian too, then $Prox(F,G)=P$.
\end{lemma}
\begin{proof}
First note that in the transformation group $(F,G)$ we have $bG=\{b\}$ and for all $x\in X$:
\[\overline{xG}=\left\{\begin{array}{lc} xG & xG{\rm \: is \: finite,} \\ xG\cup\{b\} & xG{\rm \: is \: infinite,}\end{array}\right.\]
also for $x\neq b$, $b\notin xG$. 
{\bf 1)} For $x\in F$ we have:
\begin{eqnarray*}
(x,b)\in Prox(F,G) & \Leftrightarrow & \exists\{g_\alpha\}_{\alpha\in\Gamma}\subseteq G\SP
{\displaystyle\lim_{\alpha\in\Gamma}xg_\alpha}={\displaystyle\lim_{\alpha\in\Gamma}bg_\alpha}=b
 \\
& \Leftrightarrow & b\in\overline{xG} \\
& \Leftrightarrow & b\in xG\vee(xG{\rm \: is \: infinite}) \\
& \Leftrightarrow & x\in bG\vee(xG{\rm \: is \: infinite}) \\
& \Leftrightarrow & x=b\vee(xG{\rm \: is \: infinite})
\end{eqnarray*}
Thus if $xG$ is infinite then $(x,b)\in Prox(F,G)$ which completes the proof of (1).
\\
{\bf 2)} Suppose $(x,y)\in Prox(F,G)$ we have the following cases:
\begin{itemize}
\item Case A. $x=b \vee y=b$. Without any loss of generality we may suppose $y=b$ and $(x,y)=(x,b)$. Using the proof of item (1), and $(x,b)\in Prox(F,G)$ we have ``$x=b\vee (xG$ is infinite$)$'' which leads to $(x,y)=(x,b)\in P$.
\item Case B. $xG$ and $yG$ are infinite. In this case it is clear that $(x,y)\in P$.
\item Case C. $x\neq b\wedge y\neq b\wedge (xG$ is finite or $yG$ is finite$)$. In this case we may suppose $x\neq b$ and $xG$ is finite. Since $(x,y)\in Prox(F,G)$, there exists a net $\{g_\alpha\}_{\alpha\in\Gamma}$ in $G$
such that ${\displaystyle\lim_{\alpha\in\Gamma}xg_\alpha}={\displaystyle\lim_{\alpha\in\Gamma}yg_\alpha}=:z$
thus $z\in \overline{xG}=xG\not\ni b$ so $z\neq b$ and $\{z\}$ is an open neighbourhood of $z$ (since $b$ is the unique limit point of $F$) and there exists $\alpha\in \Gamma$ with
$xg_\alpha=z=yg_\alpha$ which shows $x=y$ and $(x,y)=(x,x)\in P$
\end{itemize}
Using the above items we have $(x,y)\in P$ and  $Prox(F,G)\subseteq P$.
\\
{\bf 3)} Using (1) and (2) we have:
\begin{center}
$\{(b,x):xG{\rm \: is \: infinite}\}\cup\{(x,b):xG{\rm \: is \: infinite}\}\subseteq Prox(F,G)\subseteq 
\linebreak
 \{(x,x):x\in F\}\cup  \{(b,x):xG{\rm \: is \: infinite}\}\cup\{(x,b):xG{\rm \: is \: infinite}\}\cup 
  \{(x,y):xG{\rm \: and \:} yG{\rm \: are \: infinite}\}=P$
\end{center}
Suppose $G$ is abelian, in order to prove $Prox(F,G)=P$ we should prove for  $x,y\in F$ with infinite $xG,yG$ we have $(x,y)\in Prox(F,G)$. So consider $x,y\in F$ with infinte $xG,yG$. We have the following cases:
\begin{itemize}
\item Case I. There exists sequence $\{g_n\}_{n\geq1}$ in $G$ such that both sequences  $\{xg_n\}_{n\geq1}$ and $\{yg_n\}_{n\geq1}$
are one--to--one. In this case If $U$ is an open neighbourhood of $b$, then $F\setminus U$ is finite and 
there exists $N\geq1$ such that for all $n\geq N$ we have $xg_n,yg_n\in U$. Thus
${\displaystyle\lim_{n\geq1}xg_n}=b={\displaystyle\lim_{n\geq1}yg_n}$ and
$(x,y)\in Prox(F,G)$.
\item Case II. For each sequence $\{g_n\}_{n\geq1}$ in $G$ at least one of the sequences $\{xg_n\}_{n\geq1}$ or $\{yg_n\}_{n\geq1}$ is not one--to--one.
In this case using infiniteness of $xG$ there exists sequence $\{g_n\}_{n\geq1}$ in $G$ with infinite and one--to--one $\{xg_n\}_{n\geq1}$. If $\{yg_n:n\geq1\}$ is infinite, then there exists a subsequence $\{g_{n_i}\}_{i\geq1}$ with 
one--to--one $\{yg_{n_i}\}_{i\geq1}$, therefore
both sequences $\{xg_{n_i}\}_{i\geq1}$ and $\{yg_{n_i}\}_{i\geq1}$ are one--to--one which is in contradiction with our assumption. Thus 
$\{yg_n:n\geq1\}$ is finite, therefore $\{yg_n\}_{n\geq1}$ has a constant subsequence $\{yg_{n_i}\}_{i\geq1}$. 
Let $k_m:=g_{n_m}g_{n_1}^{-1}$ ($m\geq1$). 
Then for all $p,q\geq1$ we have:
\begin{eqnarray*}
xk_p=xk_q & \Rightarrow & xg_{n_p}g_{n_1}^{-1}=xg_{n_q}g_{n_1}^{-1} \\
& \Rightarrow & xg_{n_p}g_{n_1}^{-1}g_{n_1}=xg_{n_q}g_{n_1}^{-1}g_{n_1} \\
& \Rightarrow & xg_{n_p}=xg_{n_q} \\
& \Rightarrow & n_p=n_q \:(since \: \{xg_n\}_{n\geq1} \: is \: a \: one-to-one \: sequence) \\
& \Rightarrow & p=q
\end{eqnarray*}
moreover since $\{yg_{n_i}\}_{j\geq1}$ is a constant sequence, we have $yg_{n_p}=yg_{n_1}$ thus
$y=yg_{n_1}g_{n_1}^{-1}=yg_{n_p}g_{n_1}^{-1}=yk_p$.
\\
So $\{xk_n\}_{n\geq1}$ is a one--to--one sequence and for all $n\geq1$ we have $yk_n=y$.
Similarly there exists a sequence $\{t_n\}_{n\geq1}$ in $G$ such that
$\{yt_n\}_{n\geq1}$ is a one--to--one sequence and  $xt_n=x$ ($n\geq1$).
\\
For all $n\geq1$ we have $xk_n t_n=xt_n k_n=x k_n$ and $yk_n t_n=yt_n$, therefore
both sequences:
\[\{xk_nt_n\}_{n\geq1} (=\{xk_n\}_{n\geq1})\:\:{\rm and}\:\: \{xk_nt_n\}_{n\geq1}  (=\{yt_n\}_{n\geq1})\]
are one--to--one and infinite sequences which is in contradiction
with our assumption on $x,y$, hence this case would have not been occured.
\end{itemize}
Using the above discussion for abelian $G$ we have $(x,y)\in Prox(F,G)$ which completes the proof of (3).
\end{proof}
\begin{lemma}\label{salam20}
In infinite Fort transformation group $(F,G)$ for $x,y\in F$
and ideal $\mathcal I$ on $G$,
the following statements are equivalent:
\begin{itemize}
\item[1.] $(x,y)\in  Asym_{\mathcal I}(F,G)$,
\item[2.] for all finite subset $D$ of $F\setminus\{b\}$, we have
$\{g\in G:(xg,yg)\notin\alpha_D\}\in{\mathcal I}$,
\item[3.] for all $z\in F\setminus\{b\}$ we have
$\{g\in G:(xg,yg)\notin\alpha_{\{z\}}\}\in{\mathcal I}$.
\end{itemize}
\end{lemma}
\begin{proof}
``(1)$\Leftrightarrow$(2)'' Use definition.
\\
``(2)$\Leftrightarrow$(3)'' Use the fact that for all nonempty finite subset $D$ of $F\setminus\{b\}$
we have $\alpha_D=\mathop{\bigcap}\limits_{z\in D}\alpha_{\{z\}}$.
\end{proof}
\begin{theorem}\label{salam30}
In infinite Fort transformation group $(F,G)$ with ideal
$\mathcal I$ on $G$ we have:
\begin{eqnarray*}
Asym_{\mathcal I}(F,G) & =  & \{(x,x):x\in F\}\cup \\
& &
 \{(x,y)\in F\times F:\forall h\in G\SP stab(x)h\cup stab(y)h\in{\mathcal I}\} \cup \\
& & [\{(x,b)\in F\times F: \forall h\in G\SP stab(x)h\in{\mathcal I}\} \cup \\
& & \{(b,y)\in F\times F: \forall h\in G\SP stab(y)h\in{\mathcal I}\} \:.
\end{eqnarray*}
\end{theorem}
\begin{proof}
First note that:
\begin{itemize}
\item[$(\circledast)$] for $w\in F\setminus\{b\}$  and $z\in F\setminus wG$ we have $\{g\in G: wg=z\}=\varnothing\in\mathcal{I}$ \\
also $b\notin wG$.
\end{itemize}
For $x,y\in F\setminus\{b\}$ with $x\neq y$ we have:
\begin{eqnarray*}
(x,b)\in Asym_{\mathcal I}(F,G)
	& \Leftrightarrow & (\forall z\in F\setminus\{b\}\:
	(\{g\in G:(xg,bg)\notin\alpha_{\{z\}}\}\in\mathcal{I})) \\
& \Leftrightarrow & (\forall z\in F\setminus\{b\}\:
	(\{g\in G:(xg,b)\notin\alpha_{\{z\}}\}\in\mathcal{I})) \\
& \Leftrightarrow & (\forall z\in F\setminus\{b\}\:
	(\{g\in G:xg=z\}\in\mathcal{I})) \\
& \mathop{\Leftrightarrow}\limits^{(\circledast)} & (\forall z\in xG \:
	(\{g\in G:xg=z\}\in\mathcal{I})) \\
& \Leftrightarrow & (\forall h\in G\:
	(\{g\in G:xg=xh\}\in\mathcal{I})) \\
& \Leftrightarrow & (\forall h\in G\:
	(\{g\in G:gh^{-1}\in stab(x)\}\in\mathcal{I})) \\
& \Leftrightarrow & (\forall h\in G\:
	(stab(x)h\in\mathcal{I}))
\end{eqnarray*}
Also:
\\
{\small
$(x,y)\in Asym_{\mathcal I}(F,G)$
\begin{eqnarray*}
& \Leftrightarrow & (\forall z\in F\setminus\{b\}\:
	(\{g\in G:(xg,yg)\notin\alpha_{\{z\}}\}\in\mathcal{I})) \\
& \Leftrightarrow & (\forall z\in F\setminus\{b\}\:
	(\{g\in G:xg=z\wedge yg\neq z\}\cup
	\{g\in G:xg\neq z\wedge yg= z\}\in\mathcal{I})) \\
& \Leftrightarrow & (\forall z\in F\setminus\{b\}\:
	(\{g\in G:xg=z\}\cup
	\{g\in G:yg=z\}\in\mathcal{I})) \\
& \Leftrightarrow & (\forall z\in F\setminus\{b\}\:
	(\{g\in G:xg=z\}\in\mathcal{I}\wedge
	\{g\in G:yg=z\}\in\mathcal{I})) \\
& \Leftrightarrow & ((\forall z\in F\setminus\{b\}\:
	\{g\in G:xg=z\}\in\mathcal{I})\wedge
	(\forall z\in F\setminus\{b\}\:
	\{g\in G:yg=z\}\in\mathcal{I})) \\
& \mathop{\Leftrightarrow}\limits^{(\circledast)} & ((\forall z\in xG\:
	\{g\in G:xg=z\}\in\mathcal{I})\wedge
	(\forall z\in yG\:
	\{g\in G:yg=z\}\in\mathcal{I})) \\
& \Leftrightarrow & ((\forall h\in G\:
	\{g\in G:xg=xh\}\in\mathcal{I})\wedge
	(\forall h\in G\:
	\{g\in G:yg=yh\}\in\mathcal{I})) \\
& \Leftrightarrow & ((\forall h\in G\:
	stab(x)h\in\mathcal{I})\wedge
	(\forall h\in G\:
	stab(y)h\in\mathcal{I})) \\
& \Leftrightarrow & (\forall h\in G\:
	stab(x)h\cup	stab(y)h\in\mathcal{I})
\end{eqnarray*}
}
\end{proof}
\noindent In semigroup $S$ we say ideal $\mathcal{I}$ on $S$ is
$S-$invariant, if for all $A\in \mathcal{I}$ and $s\in S$ we have $As\in\mathcal{I}$.
So in semigroup $S$, $\mathcal{P}_{fin}(S)$ is an $S-$invariant ideal on $S$
(however for nontrivial $S$ with identity $e$, ideal $\{\{e\},\varnothing\}$
on $S$ is not $S-$invariant).
\begin{corollary}\label{salam300}
In infinite Fort transformation group $(F,G)$ with $G-$invariant ideal
$\mathcal I$ on $G$.
Then

$Asym_{\mathcal I}(F,G) =  \{(x,x):x\in F\}\cup
 \{(x,y)\in F\times F: stab(x)\cup stab(y)\in{\mathcal I}\} \cup $
 \[\{(x,b)\in F\times F: stab(x)\in{\mathcal I}\} \cup \{(b,y)\in F\times F:  stab(y)\in{\mathcal I}\} \:.\]
And:
\begin{eqnarray*}
Asym(F,G) & = &  \{(x,x):x\in F\}\cup \\
& & \{(x,y)\in F\times F: stab(x)\cup stab(y){\rm \: is \: finite}\} \cup \\
& & \{(x,b)\in F\times F: stab(x){\rm \: is \: finite}\} \cup \\
& & \{(b,y)\in F\times F: stab(y){\rm \: is \: finite}\} \:.
\end{eqnarray*}
\end{corollary}
\begin{proof}
Use Theorem~\ref{salam30},
\end{proof}
\begin{theorem}\label{salam40}
In infinite Fort transformation group $(F,G)$ suppose
$\mathcal I$ is an ideal on $S$, then:

$Prox(F,G)\setminus Asym_{\mathcal{I}}(F,G)$
\begin{eqnarray*}
& \subseteq &
	\{(x,b)\in F\times F:xG{\rm \: is \: infinite \: and \: exists\:}
	h\in G{\rm \: with \:} stab(x)h\notin\mathcal{I} \} \cup \\
	& & \{(b,x)\in F\times F:xG{\rm \: is \: infinite \: and \: exists\:}
	h\in G{\rm \: with \:} stab(x)h\notin\mathcal{I}\} \cup \\
	& & \{(x,y)\in F\times F:xG,yG{\rm \: are \: infinite \: and \: exists\:}
	h\in G{\rm \: with \:} stab(x)h\cup stab(y)h\notin\mathcal{I}\} \:.
\end{eqnarray*}
So if $\mathcal J$ is a $G-$invariant ideal on $G$, then:

$Prox(F,G)\setminus Asym_{\mathcal{J}}(F,G)$
\begin{eqnarray*}
& \subseteq &
	\{(x,b)\in F\times F:xG{\rm \: is \: infinite \: and \:} stab(x)\notin\mathcal{J} \} \cup \\
	& & \{(b,x)\in F\times F:xG{\rm \: is \: infinite \: and \:} stab(x)\notin\mathcal{J}\} \cup \\
	& & \{(x,y)\in F\times F:xG,yG{\rm \: are \: infinite \: and \:} stab(x)\cup stab(y)\notin\mathcal{J}\} \:.
\end{eqnarray*}
In particular:
\begin{eqnarray*}
Prox(F,G)\setminus Asym(F,G) & \subseteq &
	\{(x,b)\in F\times F:xG,stab(x){\rm \: are \: infinite}\} \cup \\
	& & \{(b,x)\in F\times F:xG,stab(x){\rm \: are \: infinite}\} \cup \\
	& & \{(x,y)\in F\times F:xG,yG,stab(x)\cup stab(y){\rm \: are \: infinite}\} \:.
\end{eqnarray*}
If $G$ is abelian too, we have equality in all of the above relations.
\end{theorem}
\begin{proof}
Use Lemmas \ref{salam10}, \ref{salam30} and Corollary \ref{salam300}.
\end{proof}
\begin{corollary}\label{salam50}
In infinite Fort transformation group $(F,G)$, for $S\subseteq F$
we have:
\\
1. if $S$ is an scrambled subset of $F$ module ideal $\mathcal I$ on $G$, then
\[S\setminus(\{x\in F:xG{\rm \: is \: infinite\: and \: there \: exists\:}h\in G{\rm \: with \:} stab(x)h\notin\mathcal{I}\}\cup\{b\})\]
has at most one element.
\\
2. if $S$ is an scrambled subset of $F$ modulo ideal $\mathcal J$ on $G$ and $\mathcal J$ is $G-$invariant, then
\[S\setminus(\{x\in F:xG{\rm \: is \: infinite\: and \:}stab(x)\notin\mathcal{I}\}\cup\{b\})\]
has at most one element.
\\
3. if $S$ is an scrambled subset of $F$, then
\[S\setminus(\{x\in F:xG, stab(x){\rm \: are \: infinite\:}\}\cup\{b\})\]
has at most one element.
\end{corollary}
\begin{proof}
Use Lemma~\ref{salam40}.
\end{proof}
\begin{theorem}\label{salam60}
Abelian infinite Fort transformation group $(F,G)$ is
\\
1. Li--Yorke chaotic modulo ideal $\mathcal I$ on $G$
if and only if
$H:=\{x\in F:xG$ is infinite and there exists $h\in G$ with $stab(x)h\notin\mathcal{I}\}$
is uncountable.
\\
2. Li--Yorke chaotic modulo $G-$invariant ideal $\mathcal J$ on $G$
if and only if
$H:=\{x\in F:xG$ is infinite and $stab(x)\notin\mathcal{J}\}$
is uncountable.
\\
3. Li--Yorke chaotic if and only if
$H:=\{x\in F:xG,stab(x)$ are infinite$\}$
is uncountable.
\end{theorem}
\begin{proof}
If $(F,G)$ is  Li--Yorke chaotic, then it has an uncountable
scrambled subset say $S$, by Corollary~\ref{salam50},
$S\setminus H$ is finite, so $H$ is uncountable.
\\
For infinite $H$ and abelian $G$, $H$ is an scrambled subset of $F$ by
Lemma~\ref{salam40}. So if $H$ is uncountable, then
$(F,G)$ is  Li--Yorke chaotic.
\end{proof}
\subsection*{Co--decompositions of $(F,G)$ and Li--Yorke chaos}
Now in our final notes in this section for infinite abelian group $G$, we pay attention to co--decompasability of
$(F,G)$ to Li--Yorke chatic transformation groups and co--decompasability of
$(F,G)$ to non--Li--Yorke chatic transformation groups.
\begin{corollary}\label{salam70}
In infinite abelian Fort transformation group $(F,G)$, is Li--Yorke
chaotic (modulo ideal $\mathcal I$ (on $G$)) if and only if it is co-decomposible to
Li--Yorke chaotic (modulo ideal $\mathcal I$) transformation groups.
\end{corollary}
\begin{proof}
Use Theorem~\ref{salam60}.
\end{proof}
\begin{note}
Every infinite abelian Fort transformation group $(F,G)$, is co-decomposible to
non--Li--Yorke chaotic transformation groups.
\end{note}
\begin{proof}
Suppose $(F,G)$ is an abelian Fort transformation group, then
for $\{G_\alpha:\alpha\in\Gamma\}=\{\{g^n:n\in{\mathbb Z}\}:g\in G\}$
with distinct $G_\alpha$s,
$(F,(G_\alpha:\alpha\in\Gamma))$ is a co--decomposition of
$(F,G)$ to non--Li--Yorke chaotic transformation groups.
\end{proof}
\begin{example}
For uncountable $G$ let $\mathcal{P}_{count}(G)=\{A\subseteq G:A$ is countable$\}$. Now for
$G=\mathbb{Z}\times{\mathbb R}$ and Fort space $F:={\mathbb R}\cup\{\infty\}$
with particular point $\infty$, in transformation groug $(F,G)$ with
$\infty(n,r):=\infty$ and
$x(n,r):=x+r$
$(x\in\mathbb{R},(n,r)\in \mathbb{Z}\times{\mathbb R}$
we have:
\\
1. $xG=\mathbb{R}$ for all $x\in F\setminus\{\infty\}$,
\\
2. $stab(x)=\mathbb{Z}\times\{0\}$ for all $x\in F\setminus\{\infty\}$.
\\
So by Theorem~\ref{salam60}, $(F,G)$ is Li--Yorke chaotic (modulo
$\mathcal{P}_{fin}(G)$) however it is not Li--Yorke chaotic
modulo $\mathcal{P}_{count}(G)$.
\\ $\:$ \\
As a matter of fact for transfinite cardinal numbers $\alpha,\beta$
if there exists abelian group $K$ with $\beta\leq{\rm card}(K)<\alpha$,
in group $G:=K\times\mathbb{R}$ consider two ideals
$\mathcal{I}:=\{A\subseteq G:{\rm card}(A)<\beta\}$ and
$\mathcal{J}:=\{A\subseteq G:{\rm card}(A)<\alpha\}$,
then $\mathcal{I}\subseteq \mathcal{J}$. Consider Fort space $F:={\mathbb R}\cup\{\infty\}$
with particular point $\infty$, in transformation group $(F,G)$ with
$\infty(k,r):=\infty$ and
$x(k,r):=x+r$
$(x\in\mathbb{R},(k,r)\in K\times{\mathbb R}$ we have
\\
$\bullet$ $xG=\mathbb{R}$ for all $x\in F\setminus\{\infty\}$,
\\
$\bullet$ $stab(x)=K\times\{0\}$ for all $x\in F\setminus\{\infty\}$.
\\
So by Theorem~\ref{salam60}, $(F,G)$ is Li--Yorke chaotic (modulo
$\mathcal{I}$) however it is not Li--Yorke chaotic
modulo $\mathcal{J}$.
\end{example}
\section*{Acknowledgment}
\noindent The authors would like to express their thanks to the referee 
for his/her useful comments.

\noindent {\small
{\bf Mehrnaz Pourattar},
 Department of Mathematics, Science and Research Branch, Islamic Azad University, Tehran,
Iran
({\it e-mail}: mpourattar@yahoo.com)
\\
{\bf Fatemah Ayatollah Zadeh Shirazi},
Faculty of Mathematics, Statistics and Computer Science,
College of Science, University of Tehran ,
Enghelab Ave., Tehran, Iran \linebreak
({\it e-mail}: fatemah@khayam.ut.ac.ir)}

\end{document}